\documentclass[12pt]{amsart}
\DeclareFontFamily{OML}{script}{}
\DeclareFontShape{OML}{script}{m}{it}
{ <5-20> rsfs10 }{}
\DeclareMathAlphabet{\mathscript}{OML}{script}{m}{it}

\renewcommand{\mathcal}[1]{{\mathscript #1}\hspace{0.2ex}}
\usepackage{color}
\ifx\red\undefined
\newcommand{\red}{\color{red}}
\fi

\usepackage{graphicx,graphics}
\usepackage{pifont}
\usepackage{amsthm}

\usepackage{amscd}
\usepackage{amsmath}
\usepackage{latexsym}
\usepackage{amsfonts}
\usepackage{amssymb}
\usepackage{color}

\ifx\text\undefined
\newcommand{\text}{\mbox}
\fi
\ifx\operatorname\undefined
\newcommand{\operatorname}{\mathop}
\fi
\newcommand\be{\begin{equation}}
\newcommand\ee{\end{equation}}
\newcommand\bea{\begin{eqnarray}}
\newcommand\eea{\end{eqnarray}}
\newcommand\beaa{\begin{eqnarray*}}
\newcommand\eeaa{\end{eqnarray*}}

\newenvironment{eqa}{\begin{equation}%
  \begin{array}{rcl}}{\end{array}\end{equation}}
\newcommand\beqa{\begin{eqa}}
\newcommand\eeqa{\end{eqa}}

\numberwithin{equation}{section}

\newtheorem{thm}{Theorem}[section]

\newtheorem{lem}{Lemma}[section]
\newtheorem{defn}{Definition}[section]
\newtheorem{rem}{Remark}[section]
\newtheorem{prop}{Proposition}[section]

\newcommand{\void}[1]{}
\newcommand\vep{{\varepsilon}}

\newcommand{\dif}{\mathrm{d}}
\newcommand{\me}{\mathrm{e}}

\newcommand{ \D }{ \displaystyle }

\begin{document}

\title[Classification of certain qualitative properties]{Classification of certain qualitative properties of solutions for the quasilinear parabolic equations}

\author{Yan Li, Zhengce Zhang}
\date{\today}
\address{School of Mathematics and Statistics, Xi'an Jiaotong University,
Xi'an, 710049, P. R. China}
\email{zhangzc@mail.xjtu.edu.cn, liyan1989@stu.xjtu.edu.cn}
\author{Liping Zhu}
\address{College of Science, Xi'an University of Architecture \& Technology,
Xi'an, 710055, P. R. China}
\email{78184385@qq.com}
\thanks{Corresponding author: Zhengce Zhang}

\begin{abstract}
In this paper, we mainly consider the initial boundary problem for a
quasilinear parabolic equation
\[
u_t-\mathrm{div}\left(|\nabla u|^{p-2}\nabla
u\right)=-|u|^{\beta-1}u+\alpha|u|^{q-2}u,
\]
where $p>1,\beta>0$, $q\geq1$ and $\alpha>0$. By using Gagliardo-Nirenberg type inequality, energy method and comparison principle, the phenomena of blowup and extinction are
classified completely in the different ranges of reaction exponents.
\end{abstract}
\maketitle
\section{Introduction}

In this paper, the following initial boundary problem is considered:
\begin{equation}\label{IBP}
\left\{\begin{aligned}
u_t-\Delta_p u&=-|u|^{\beta-1}u+{\alpha|u|^{q-2}u},&x\in\Omega,t>0,\\
u&=0,&x\in\partial\Omega,t>0,\\
u(x,0)&=u_0(x),&x\in\Omega,\\
\end{aligned}\right.
\end{equation}
where $\Omega\subset\mathbb{R}^N\ (N\geq1)$ is a smoothly bounded domain and
$p>1, \beta>0, q\geq1,\alpha>0$. The operator $\Delta_p$ is defined as follows:
\[
\Delta_p u:=\mathrm{div}\left(|\nabla u|^{p-2}\nabla u\right).
\]
We also suppose that $u_0(x)\geq 0, u_0(x)\not\equiv0, u_0(x)\in
W_0^{1,p}(\Omega)\cap L^\infty(\Omega)$.

Problem (\ref{IBP}) arises in the theory of nonstationary filtration of
non-Newtonian (or dilatant) fluids and combustion of solid fuels. The term
$-|u|^{\beta-1}u$, which is negative as we can prove later that $u\geq0$, is
called a singular absorption term for $\beta<0$ or a strong absorption one
for $0<\beta<1$ or a weak absorption one for $\beta>1$. {$\alpha|u|^{q-2}u$} is an inner
source term. It has been known for many years that the term $-|u|^{\beta-1}u$
with $\beta>0$ may lead to finite time extinction, i.e. \emph{there exists a
$T\in (0, +\infty)$ such that $u(x,t)$ is nontrivial for $t\in[0,T)$ and
$u(x,t)\equiv0$ for $t\in[T,+\infty)$ a.e. in $x\in\Omega$.} On the other
hand, $f(x,u)$ may lead to finite time blowup. However, if the two terms
appear simultaneously in the first equation of (\ref{IBP}), then the
solutions will exhibit complicated properties {which will be studied later}. To be specific, both blowup
and extinction can occur under some suitable conditions.

As the operator $\Delta_p$ is degenerate for $p>2$ and is singular for $1<p<2$, it's impossible to consider the classical solution of \eqref{IBP} generally. However, the concept of weak solution is enough for our study. For the local existence of weak solution of \eqref{IBP}, there are various methods can be applied such as approximation by regular solution \cite{AA,JNZ}, fixed point method \cite{YJ}, the method of extension of semigroup \cite{GV} and the developed Faedo-Galerkin method \cite{AS,ASS,GSS}.

As soon as the local existence is established, one may ask whether the weak
solution is global or not. Moreover, we are eager to know when the solution
is global in time and when it blows up in finite time. For the global
solution, we also want to know whether it will become zero in finite time or
not.

The phenomenon of finite time blowup was first considered by H. Fujita \cite{HF} in
1966. Since then, many people devoted themselves to this problem. The main
equation they studied is the heat equation of the form $u_t-\Delta u=|u|^{p-1}u$
in bounded or unbounded smooth domain in $\mathbb{R}^N$. The theory of
blowup for heat equation is already developed, we refer the reader to
\cite{HM,LP,PQ,QS,SW} and the references therein. While for the $p$-Laplacian
equations {of the form $u_t-\Delta_pu=f(x,t,u,\nabla u)$}, there are still many problems worth studying, such as the blowup
rate, {the blowup time estimate,} the asymptotic behavior of blowup solutions, the blowup criteria and so
on. Some related results can be found in \cite{GP,GSS,LX,QBZ,YJ,ZL,ZL1,JNZ,JZDY,ZLR}
and the references therein. To be specific, in \cite{GSS,LX,QBZ,YJ,JNZ}, criteria for the
finite time blow-up to occur were established {in bounded domain} for different kinds of source terms and
values of $p$. {Generally speaking, finite time blowup may occur if $f(x,t,s,\vec{r})$ grows faster than $s^{p-1}$ ($p>2$) or $s$ ($1<p<2$) ($q=p-1$ or $q=1$ is called the critical blowup exponent) when $s\rightarrow\infty$ and the initial data is large enough. } In \cite{GP}, Galaktionov and Posashkov studied the blowup
set for the equation $u_t-\hbox{div}\left(|\nabla u|^\sigma\nabla
u\right)=u^\beta$
with $\sigma>0,\beta>1$ and $x\in\mathbb{R}^N$. {They proved that the radial solution will blow up at $|x|=0$. For the blowup time estimate, Zhou and Yang \cite{JZDY} considered the equation $u_t-\mathrm{div}\left(|\nabla u|^{m-2}\nabla u\right)=|u|^{p(x)-1}u$ with Dirichlet boundary condition on bounded domains. They obtained a upper bound of the blowup time for some suitable conditions on $m,p(x)$ and initial data.} and  {Zhao and Liang \cite{ZLR} considered a Cauchy problem $u_t-\Delta_p u=u^q$ in the radial situation and obtained the blowup rate upper bound is of the order $(T-t)^{-1/(q-1)}$ for $q>p-1$.} In our latest
papers \cite{ZL,ZL1}, we considered the equation
$u_t-\Delta_pu=\lambda u^m+\mu|\nabla u|^q$
with $p>2$ and $\lambda\mu<0$, and proved that $u$ will blow up in finite
time in the $L^\infty$-norm sense if $\lambda>0,\mu<0$ and
$m>\max\{p-1,q\},q\leq p/2$. {For the blowup of more general p-Laplacian equations, there are also some important results. In \cite{WZW,YYZ}, the Fujita exponent for equations with weighted source of the form
\begin{equation*}
\frac{\partial u}{\partial t}=\mathrm{div}\left(|\nabla u|^{p-2}\nabla u\right)+k\frac{1}{|x|^2}|\nabla u|^{m-1}\nabla u\cdot x+|x|^\lambda u^p
\end{equation*}
  were studied. In \cite{CLM1,JZ},  the global existence, blowup  and the blowup point of solutions for the doubly degenerate equations, i.e. equations with $\mathrm{div}\left(|\nabla u^m|^{p-2}\nabla u^m\right)$ were carefully studied.}

Finite time extinction is another important property of solutions of
evolution equations. Since Kalashnikov first brought in the concept of
extinction in 1974, it has attracted many mathematicians' interests and most
of them focused on the fast diffusive equations, see
\cite{Diaz,ED,FWL,FL,FX,YGG,KYC,MW1,YJ} for examples.
Moreover, in \cite{JLV}, the homogeneous $p$-Laplacian equation $u_t=\Delta_pu$ with $p>1,x\in\mathbb{R}^N$ was studied. It was shown that extinction can happen if and only if $1<p\leq p_c=2N/(N+1)$. In \cite{YJ}, Yin and Jin
considered the equation $u_t-\Delta_pu=\lambda u^q$ with $x\in\Omega$ and $1<p<2$. They proved that $q=p-1$ is
the critical extinction exponent. In \cite{YGG}, Gu considered the
$p$-Laplacian equation { $u_t-\Delta_pu=-|u|^{\beta-1}u$ }with $p>1$. In that paper, the
conditions for extinction to occur were obtained for any $p>1$ while the
non-extinction condition was obtained only for $p\geq2$. For the equation
with absorption and source terms, i.e.
$u_t-\Delta_pu=\lambda u^q-\beta u^k$
with $1<p<2$ and $0<q, k<1$, it was showed in \cite{FWL} that the solution
will exhibit extinction phenomenon under the assumptions that $u_0(x)$ or
$\lambda$ is small enough and that $\beta$ is large enough. {In \cite{JYZ,QBZ}, the extinction phenomenon for p-Laplacian with Neumann boundary data and nonlocal absorption term were studied.}

In this paper, we will deal with problem (\ref{IBP}) for any
$p>1$. In Section 2, we will give some basic concepts and a weak comparison
principle. Section 3 is devoted to the existence of the weak solution for
problem (\ref{IBP}) in a general case. The extinction phenomenon will be
discussed in Section 4. At last, we will give some blowup results under
different conditions for $u_0(x)$ and $p, \beta, q, \alpha$.

\section{Preliminaries}\setcounter{equation}{0}
Before giving the definition of weak solution, we bring in the following
function space:
\begin{equation}
\mathbb{V}:=\left\{v\in
L^p\left(0,T;W_0^{1,p}(\Omega)\right)\Big|\partial_tv\in
L^{p'}\left(0,T;W^{-1,p'}(\Omega)\right)\right\}.
\end{equation}

Now, let us introduce the definition of weak solution of (\ref{IBP}).
\begin{defn}\label{WS}
Let $Q_T=\Omega\times(0,T), S_T=\partial\Omega\times(0,T),\partial
Q_T=S_T\cup\left\{\overline{\Omega}\times\{0\}\right\}$. A function $u\in
\mathbb{V}\cap C\left(0,T;L^2(\Omega)\right)$ is called a weak solution of
(\ref{IBP}) if it satisfies:

1. for every nonnegative test-function $\varphi\in \mathbb{V}\cap
C\left(0,T;L^2(\Omega)\right)$,
\begin{equation}\label{ws}
\iint_{Q_T}\left(\partial_tu\varphi+|\nabla u|^{p-2}\nabla
u\cdot\nabla\varphi\right)\ \dif x\dif t=-\iint_{Q_T}
\left(|u|^{\beta-1}u-{\alpha|u|^{q-2}u}\right)\varphi\ \dif x\dif t.
\end{equation}

2. $u(x,0)=u_0(x)$ for a.e. $x\in\Omega$.

Moreover, if we replace ``$=$'' in (\ref{ws}) by``$\leq$''(``$\geq$'') and
assume that $u(x,0)\leq(\geq) u_0(x), u(x,t)|_{x\in\partial\Omega}\leq(\geq)0$,
then the corresponding solution is called a sub-(sup-) solution.
\end{defn}

For the weak solution of (\ref{IBP}), we have the following weak comparison
principle. Some similar results can be found in \cite{AA,LX,YJ,ZL}.

\begin{prop}\label{CP}
Suppose that $u,v$ are weak sub- and sup- solutions of (\ref{IBP})
respectively. If $u$ and $v$ are locally bounded, then $u\leq v$ a.e. in
$Q_T$.
\end{prop}
\begin{proof} Let $\varphi=\max\{u-v,0\}$, then
$\varphi(x,0)=0,\varphi(x,t)|_{x\in\partial\Omega}=0$. By Definition
\ref{WS}, $\varphi(x,t)$ satisfies:
\begin{equation}\label{cp}
\begin{split}
&\iint_{Q_T} \partial_t\varphi\varphi\ \dif x\dif t+\underbrace{\iint_{Q_T}
\left(|\nabla u|^{p-2}\nabla u-|\nabla v|^{p-2}\nabla v\right)(\nabla
u-\nabla v)\ \dif x\dif t}_{\mathcal{M}}\\
\leq&-\iint_{Q_T}\left(|u|^{\beta-1}u-|v|^{\beta-1}v\right)\varphi\ \dif
x\dif t+\alpha\iint_{Q_T} ({|u|^{q-2}u-|v|^{q-2}v})\varphi\ \dif x\dif t\\
\leq&-\underbrace{\iint_{Q_T}\left(|u|^{\beta-1}u-|v|^{\beta-1}v\right)\varphi\
\dif x\dif t}_{\mathcal{A}}+L\iint_{Q_T} \varphi^2\ \dif x\dif t
\end{split}
\end{equation}
where $L$ is a constant depending on the sup-norms of $u$ and $v$.

Let us now estimate terms $\mathcal{M}$ and $\mathcal{A}$ appearing in
(\ref{cp}). By the monotone inequality (see \cite{PL}), we have
$\mathcal{M}\geq0$ for any $p>1$. For term $\mathcal{A}$, by the fact that
\begin{equation}\label{termA}\left\{
\begin{array}{l}
|u|^{\beta-1}u-|v|^{\beta-1}v=u^\beta-v^\beta>0,\ \mbox{if}\ u>v>0,\\[2mm]
|u|^{\beta-1}u-|v|^{\beta-1}v=u^\beta+|v|^\beta>0,\ \mbox{if}\ u>0>v,\\[2mm]
|u|^{\beta-1}u-|v|^{\beta-1}v=-|u|^\beta+|v|^\beta>0,\ \mbox{if}\ 0>u>v,\\[2mm]
\end{array}\right.
\end{equation}
we have $\mathcal{A}\geq0$.

Following the discussion above, we have
\begin{equation}
\frac{1}{2}\int_\Omega \varphi^2\ \dif x\leq L\iint_{Q_T} \varphi^2\ \dif x\dif t.
\end{equation}
By Gronwall's inequality, we have $\int_\Omega \varphi^2\ \dif x=0$. This
implies that $\varphi=0$ a.e. $x\in\Omega$, i.e. $u\leq v$ a.e. $(x,t)\in
Q_{T}$.
\end{proof}
\section{Existence of weak solution}\setcounter{equation}{0}
In this section, we will establish the local existence and global existence
of weak solutions of (\ref{IBP}). Analogous to the proofs in \cite{AS,ASS,GSS} and the compactness results in \cite{JS}, we have the following local existence of bounded weak solution for \eqref{IBP}.
\begin{thm}\label{HOR}
Suppose that $u_0\in W_0^{1,p}(\Omega)\cap L^\infty(\Omega),u_0\geq0,u_0\not\equiv0$ a.e. in
$\Omega$ and that $q\geq1$. Then there exists a $T^*=T^*(u_0)>0$ such that for $0<T<T^*$ \eqref{IBP} admits a solution
\begin{equation}
u\in \mathbb{U}:=\left\{u\in L^\infty\left(0,T;W_0^{1,p}(\Omega)\right)\cap
L^\infty(Q_T)\ \Big|\ \partial_t u\in L^2(Q_T)\right\}.
\end{equation}
Moreover, $0\leq u\leq M$ a.e. in $Q_T$ for some $M$ depending on $u_0(x)$.
\end{thm}

Next, we will give some results focusing on the global existence of the weak
solution for (\ref{IBP}).

Denote by $\Lambda_1>0$ the first eigenvalue of the $p$-Laplacian operator
with homogeneous Dirichlet boundary condition, i.e.
\begin{equation}
\Lambda_1:=\inf\left\{\int_\Omega |\nabla u|^p\ \dif x\ \Big|\ u\in
W_0^{1,p}(\Omega),\int_\Omega |u|^p\ \dif x=1\right\}.
\end{equation}

\begin{thm}[Global existence]\label{GE}
Let $u_0(x)\in W^{1,p}_0(\Omega)\cap L^\infty(\Omega),u_0(x)\geq0$ and one of the following
conditions is satisfied

$(\mathrm{i})\ q=p,\alpha<\Lambda_1$.

$(\mathrm{ii})\ q=p=\beta+1,\alpha<\Lambda_1+1$.

$(\mathrm{iii})\ 2<p\leq q<\beta+1$.

{$(\mathrm{iv})\ q<p$.}

Then the solution of (\ref{IBP}) is globally in time bounded, i.e. there
exists a constant $M$ depends only on $p,q,\beta,\Lambda_1,\alpha,u_0,\Omega$
such that for every $T>0,0\leq u\leq M$.
\end{thm}

\begin{proof} Case $(\mathrm{i})$. Let
$\widetilde{\Omega}\subset\mathbb{R}^N$ be a smooth domain which satisfies:
$\Omega\subset\subset\widetilde{\Omega}$. Denote by $\phi$ and
$\Lambda_1(\widetilde{\Omega})$ the first eigenfunction and the first
eigenvalue related to the following Dirichlet problem:
\begin{equation}
-\Delta_p\phi=\Lambda_1(\widetilde{\Omega})|\phi|^{p-2}\phi\ \mbox{in}\
\widetilde{\Omega},\ \phi=0\ \mbox{on}\ \partial\widetilde{\Omega},\
\int_{\widetilde{\Omega}} |\phi|^p\ \dif x=1.
\end{equation}
Then by \cite[Lemma 1.1]{LX}, we know that $\phi>0$ in $\widetilde{\Omega}$
and that $\Lambda_1(\widetilde{\Omega})<\Lambda_1(\Omega)$. Moreover, by
\cite[Theorem 3.2]{LYI}, $\Lambda_1(\widetilde{\Omega})$ continuously depends
on $\widetilde{\Omega}$ and
$\Lambda_1(\widetilde{\Omega})\rightarrow\Lambda_1(\Omega)$ as
$\widetilde{\Omega}\rightarrow\Omega$ in the Hausdorff complementary
topology. Thus, we can
choose a suitable $\widetilde{\Omega}$ and $\theta>0$ such that
$\alpha\leq\Lambda_1(\widetilde{\Omega})\leq\Lambda_1(\Omega)$.
Let $\Phi=K\phi\geq K\mu\geq\|u_0\|_{L^\infty(\Omega)}$
with $\mu=\inf_\Omega\phi>0$. Then a simple
calculation shows that for every nonnegative test-function
$\varphi\in\mathbb{V}\cap (0, T; L^2 (\Omega))$
\begin{equation}
\begin{split}
\iint_{Q_T}
\partial_t\Phi\varphi+|\nabla\Phi|^{p-2}\nabla\Phi\cdot\nabla\varphi\ \dif
x\dif t
&\geq\Lambda_1(\widetilde{\Omega})\iint_{Q_T} \Phi^{p-1}\varphi\ \dif x\dif
t\\
&\geq\alpha\iint_{Q_T} \Phi^{p-1}\varphi\ \dif x\dif t
\end{split}
\end{equation}
This implies that $\Phi$ is a sup-solution of (\ref{IBP}). Then by
Proposition \ref{CP}, we have $0\leq u\leq\Phi$ a.e. in $Q_T$. We can also
see from the construction of $\Phi$ that it's independent of $t$ which
enables us to continue the procedure above on any time interval $[T,T']$.
Then, we can assert that the solution of (\ref{IBP}) is globally in time
bounded.

The proof of Case $(\mathrm{ii})$ is same as the one of Case $(\mathrm{i})$.

Case $(\mathrm{iii})$. { Without loss of generality, we assume $\alpha=1$}, the method below is still valid for the general case
with a little modification. Denote by $\rho(\Omega)$ the diameter of
$\Omega$, then we can easily know that $\rho(\Omega)<\infty$ as $\Omega$ is
bounded. Let $\vep\in(0,1)$ satisfies: there exists a ball of radius $\vep$
belonging to $B(\cdot,\rho(\Omega)+1)\cap\Omega^c$. For any $a\in\Omega$, let
$x_a$ satisfies:
\begin{equation}\label{xa}
B(x_a,\vep)\subset B(x_a,\rho(\Omega)+1)\cap\Omega^c, |x_a-a|<\rho(\Omega)+1.
\end{equation}
Let
\begin{equation}\label{DOV}
V(x,t)=L\me^{\sigma r},\ r=|x-x_a|,\ x\in\Omega.
\end{equation}
Define: $\mathcal{L}_p v:=v_t-\Delta_pv-v^{q-1}+v^\beta$, then $V(x,t)$
satisfies
\begin{equation}
\mathcal{L}_pV=-(p-1)(L\sigma)^{p-1}\me^{(p-1)\sigma
r}-\frac{N-1}{r}(L\sigma)^{p-1}\me^{(p-1)\sigma r}-L^{q-1}\me^{(q-1)\sigma
r}+L^\beta\me^{\beta\sigma r}.
\end{equation}
In order to derive that $\mathcal{L}_pV\geq 0$, we need to choose suitable $\sigma$ and $L$
such that
\begin{equation}\label{Lpv}
(p-1)\sigma^p+\frac{N-1}{r}\sigma^{p-1}\leq
L^{\beta+1-p}\me^{(\beta+1-p)\sigma r}-L^{q-p}\me^{(q-p)\sigma r}.
\end{equation}
By (\ref{xa}) and (\ref{DOV}), we know that $\vep\leq r<\rho(\Omega)+1$. Then
if we want (\ref{Lpv}) to be satisfied, it's sufficient that
\begin{equation}
(p-1)\sigma^p+\frac{N-1}{\vep}\sigma^{p-1}+L^{q-p}\me^{(q-p)\sigma(\rho(\Omega)+1)}\leq
L^{\beta+1-p}.
\end{equation}
If $q>p$, let $\sigma$ and $L$ satisfy
\begin{equation}
\sigma=\frac{1}{(q-p)(\rho(\Omega)+1)},
L=\max\left\{(2\me)^\frac{1}{\beta+1-q},\left(2\left((p-1)\sigma^p+\frac{N-1}{\vep}\sigma^{p-1}\right)\right)^\frac{1}{\beta+1-p}\right\}.
\end{equation}
While if $q=p$, let $\sigma$ and $L$ satisfy
\begin{equation}
\sigma=1,
L=\max\left\{2^\frac{1}{\beta+1-q},\left(2\left(p-1+\frac{N-1}{\vep}\right)\right)^\frac{1}{\beta+1-p}\right\}.
\end{equation}
Then there holds $\mathcal{L}_pV\geq0$. If we assume furthermore that
$L\geq\|u_0\|_{L^\infty(\Omega)}$,
then $V(x,0)\geq u_0(x)$. Thus, we have proved that $V(x,t)$ is a
super-solution of (\ref{IBP}). By Proposition \ref{CP}, we have
\begin{equation}\label{ub}
u(x,t)\leq L\me^{\sigma(\rho(\Omega)+1)}<\infty.
\end{equation}
Notice that the right hand side of (\ref{ub}) is in fact independent of
$t$, which enable us to continue the procedure above in any time interval
$[T,T']$. Hence, we can conclude that $u(x,t)$ is globally in time bounded.

{In the case $q<p$, by
Young's inequality, there exists a small $\gamma>0$ such that
$\alpha|s|^{q-2} s\leq\gamma|s|^{p-1}+C(\gamma)$. Then the
conclusion follows from the same procedure as above.}
\end{proof}

\section{Finite time extinction and Decay}\setcounter{equation}{0}
Before proving our main results, we first introduce the following
Gagliardo-Nirenberg type inequality which can be found in \cite{ED,GSS} and
the references therein.

\begin{lem}\label{GN}
Let $1<p<+\infty$ and $r\in[\beta+1,+\infty)$ if $p\geq N$, and
$r\in\left[\beta+1,\frac{Np}{N-p}\right]$ if $p<N$. Then there exists a
constant $C>0$, depending only on $p,r,N,\beta$ and $|\Omega|$, such that for
every $u\in W_0^{1,p}(\Omega)$
\begin{equation}\label{gn}
\|u\|_{L^r(\Omega)}\leq C\|\nabla
u\|^\theta_{L^p(\Omega)}\|u\|^{1-\theta}_{L^{\beta+1}(\Omega)}\ \mbox{with}\
\theta=\frac{\frac{1}{\beta+1}-\frac{1}{r}}{\frac{1}{N}-\frac{1}{p}+\frac{1}{\beta+1}}\in{[0,1]}.
\end{equation}
\end{lem}

\begin{rem}\label{theta}
We can see from the expression of $\theta$ with $r>\beta+1$ that
\begin{equation}
\theta<\frac{\frac{1}{\beta+1}-\frac{1}{r}}{-\frac{1}{p}+\frac{1}{\beta+1}}
\end{equation}
and that
\begin{equation}
r\left(\frac{\theta}{p}+\frac{1-\theta}{\beta+1}\right)>1
\end{equation}
which will play an important role in establishing a desired ordinary
differential inequality later.
\end{rem}

\subsection{Finite time extinction}
The following theorem deals with the finite time extinction.
\begin{thm}\label{FTE}
Let $\beta+1\leq q\leq p$ and $\beta<\min\{1,p-1\}$. Assume additionally that
 $\alpha<\min\{1,\Lambda_1\}$. Then there
exists a finite time $T^*>0$, such that $u=0$ a.e. in $\Omega$ for $t\geq
T^*$.
\end{thm}

\begin{proof}{By Theorem \ref{GE}, $u$ exists globally in time.} Let $y(t)=\|u\|^2_{L^2(\Omega)}$. Then it satisfies:
\begin{equation}\label{ODE}
\frac{1}{2}y'(t)+\int_\Omega |\nabla u|^p\ \dif x=\alpha\int_\Omega u^q\ \dif
x-\int_\Omega u^{\beta+1}\ \dif x.
\end{equation}
By the assumption that $\beta+1\leq q\leq p$, we have
\begin{equation}\label{uq}
\begin{split}
\int_\Omega u^q\ \dif x&=\int_{\Omega\cap\{u\geq1\}} u^q\ \dif
x+\int_{\Omega\cap\{u\leq1\}} u^q\ \dif x
\leq\int_{\Omega\cap\{u\geq1\}} u^p\ \dif x+\int_{\Omega\cap\{u\leq1\}}
u^{\beta+1}\ \dif x\\
&\leq\int_\Omega\left(u^p+u^{\beta+1}\right)\ \dif x
\leq\frac{1}{\Lambda_1}\int_\Omega |\nabla u|^p\ \dif x+\int_\Omega
u^{\beta+1}\ \dif x,
\end{split}
\end{equation}
where we used the Poincar\'{e}'s inequality
$\Lambda_1\|u\|^p_{L^p(\Omega)}\leq\|\nabla u\|^p_{L^p(\Omega)}$. Combining
(\ref{ODE}) with (\ref{uq}), we find that for
\begin{equation}
D=\left\{\begin{array}{ll}
1-\alpha\max\left\{\D\frac{1}{\Lambda_1},1\right\},&\mbox{if}\
\beta+1<q<p,\\[2mm]
1-\alpha,&\mbox{if}\ \beta+1=q<p,\\[2mm]
1-\D\frac{\alpha}{\Lambda_1},&\mbox{if}\ \beta+1<q=p,
\end{array}\right.
\end{equation}
there holds
\begin{equation}\label{ODI}
\frac{1}{2}y'(t)+D\int_\Omega \left(|\nabla u|^p+u^{\beta+1}\right)\ \dif
x\leq 0.
\end{equation}

Our next goal is to obtain the following differential inequality from
(\ref{ODI}):
\begin{equation}\label{FODI}
y'(t)+Ky(t)^\gamma\leq0,\ \mbox{with}\ K>0,0<\gamma<1.
\end{equation}
Integrating (\ref{FODI}) with $t$:
\begin{equation}
y(t)\leq\left(y^{1-\gamma}(0)-K(1-\gamma)t\right)^\frac{1}{1-\gamma}
\end{equation}
which implies
\begin{equation}
y(t)\rightarrow0\ \mbox{as}\ t\rightarrow
T^*:=\frac{y^{1-\gamma}(0)}{K(1-\gamma)}.
\end{equation}
Thus, the finite time extinction for the solution of (\ref{IBP}) is proved.

To obtain (\ref{FODI}), we divided our proof into two parts:
$p>\frac{2N}{N+2}$ and $1<p<\frac{2N}{N+2}$.

($\mathrm{i}$). If $p\geq\frac{2N}{N+2}$, then $\frac{Np}{N-p}\geq2$ for $p<N$ which implies
that we can choose $r=2$ in (\ref{gn}). While if $p\geq N$, then
$r\in[\beta+1,+\infty)$ which enables us to set $r=2$ in (\ref{gn}). In both
cases, we can obtain
\begin{equation}
\begin{split}
\|u\|_{L^2(\Omega)}&\leq C\|\nabla
u\|^\theta_{L^p(\Omega)}\|u\|^{1-\theta}_{L^{\beta+1}(\Omega)}=C\left(\int_\Omega
|\nabla u|^p\ \dif x\right)^\frac{\theta}{p}\left(\int_\Omega u^{\beta+1}\
\dif x\right)^\frac{1-\theta}{\beta+1}\\
&\leq C\left(\int_\Omega \left(|\nabla u|^p+u^{\beta+1}\right)\ \dif
x\right)^{\frac{\theta}{p}+\frac{1-\theta}{\beta+1}}
\end{split}
\end{equation}
from (\ref{gn}) with $r=2$.
Then
\begin{equation}\label{odi}
C^{-2}y(t)\leq\left(\int_\Omega \left(|\nabla u|^p+u^{\beta+1}\right)\ \dif
x\right)^{2\left(\frac{\theta}{p}+\frac{1-\theta}{\beta+1}\right)}.
\end{equation}
Combining (\ref{odi}) with (\ref{ODI}), we can obtain (\ref{FODI}) with
\begin{equation}
\frac{1}{\gamma}=2\left(\frac{\theta}{p}+\frac{1-\theta}{\beta+1}\right)>1,K=2DC^{-2\gamma}.
\end{equation}

($\mathrm{ii}$). If $1<p<\frac{2N}{N+2}$, let $2>r\in(\beta+1,\frac{Np}{N-p}]$ and
$M=\|u\|_{L^\infty(Q_T)}$. Then we have
\begin{equation}
y(t)=\|u\|_{L^2(\Omega)}=\int_\Omega u^{2-r}u^r\ \dif x\leq
M^{2-r}\|u\|_{L^r(\Omega)}^r.
\end{equation}
By (\ref{gn}) with $r\in(\beta+1,2)$, there
holds
\begin{equation}\label{odi2}
\begin{split}
y(t)&\leq M^{2-r}\left(C\|\nabla
u\|^\theta_{L^p(\Omega)}\|u\|^{1-\theta}_{L^{\beta+1}(\Omega)}\right)^r\\
&\leq
M^{2-r}\left(D^{\frac{\theta}{p}+\frac{1-\theta}{\beta+1}}\right)^{-r}C^r\left(D\int_\Omega
\left(|\nabla u|^p+u^{\beta+1}\right)\ \dif
x\right)^{r\left(\frac{\theta}{p}+\frac{1-\theta}{\beta+1}\right)}.
\end{split}
\end{equation}
Combining (\ref{odi2}) with (\ref{ODI}), we can derive (\ref{FODI}) with
\begin{equation}
\frac{1}{\gamma}=r\left(\frac{\theta}{p}+\frac{1-\theta}{\beta+1}\right)>1,K=2DM^{\gamma(r-2)}C^{-r\gamma}.
\end{equation}
\end{proof}
\begin{rem}
In the case $1<p<2$, Fang, Wang and Li \cite{FWL} obtained some similar
extinction results. The results there needed stronger conditions for the
coefficients of absorption and source terms. Moreover, the initial data was
also been chosen small enough. However, our results hold for any nontrivial initial
data and some $\alpha$ which needn't to be sufficiently small. Besides, our
proof is also simpler.
\end{rem}
Different from Theorem \ref{FTE}, the following theorem shows that finite
time extinction can also occur for $q>p$ and $1<p<2$ with small initial data.
\begin{thm}\label{FTE2}
Assume that $q>p,1<p<2$, then the solution of
(\ref{IBP}) will vanish at finite time provided the initial data is small
enough.
\end{thm}
\begin{proof} The proof here is same as the one in {\cite[Theorem 4.1]{YJ}}, we omit it.
\end{proof}
\subsection{Decay}
Let us now consider the decay of the solution.

\begin{thm}\label{decay}
Assume that $\beta\geq1$ and {$p\geq2$}, then the solution of (\ref{IBP}) will not
extinguish in finite time. Assume additionally $\beta\leq q-1$, then there
exists a constant $\epsilon>0$, such that if $u_0\geq0$ and
$\|u_0\|_{L^\infty(\Omega)}<\epsilon$,
then the solution will decay to zero as $t\rightarrow+\infty$. Moreover, we
have the following estimates:
\begin{equation}\label{Decay}
\left\{\begin{aligned}
&0\leq u\leq C_1(t+C_2)^{-\gamma},\gamma=\frac{1}{\beta-1},\mbox{for}\
1<\beta\leq q-1;\\
&0\leq u\leq C_3\me^{-C_4t},\mbox{for}\ \beta=1,q\geq 2.
\end{aligned}\right.
\end{equation}
The constants $C_i,i=1,2,3,4$ appeared above depend on $q,\beta,\alpha$.
\end{thm}
\begin{proof}
By \cite[Theorem 3.3]{YGG}, we know that the solution of
\begin{equation}
\left\{
\begin{aligned}
v_t-\Delta_pv&=-|v|^{\beta-1}v,&x\in\Omega,t>0,\\
v&=0,&x\in\partial\Omega,t>0,\\
v(x,0)&=u_0(x),&x\in\Omega.
\end{aligned}
\right.
\end{equation}
will not extinguish in finite time if $p\geq2,\beta\geq1,u_0(x)\in
W_0^{1,p}(\Omega)\cap L^\infty(\Omega),u_0(x)\not\equiv0$. As was shown in
Theorem \ref{HOR}, $u\geq 0$. Thus, $v$ is a sub-solution of \eqref{IBP}. By the comparison principle, $u$ will not extinguish in finite time.

Let us now consider the decay of the solution of (\ref{IBP}). For
convenience, we define $\mathcal{L}_p$ as:
{$\mathcal{L}_p\varphi=\varphi_t-\Delta_p\varphi-\alpha|\varphi|^{q-2}\varphi+
|\varphi|^{\beta-1}\varphi$.}

If $1<\beta\leq q-1$, let
\begin{equation}
w(x,t)=C_1(t+C_2)^{-\gamma},\gamma=\frac{1}{\beta-1},
\end{equation}
where $C_1,C_2>0$ are constants to be decided later. By a direct computation,
we have
\begin{equation}
\mathcal{L}_pw=(t+C_2)^{-\gamma-1}\left(-\gamma C_1+C_1^\beta-\alpha
C_1^{q-1}(t+C_2)^{-(q-1-\beta)\gamma}\right).
\end{equation}
If $\beta<q-1$, let $C_1,C_2>0$ satisfy:
$(2\gamma)^\gamma\leq C_1\leq(2\alpha)^\frac{1}{\beta-q+1}C_2^\gamma$,
then we have $\mathcal{L}_pw\geq0$. Assume additionally that
$\|u_0\|_{L^\infty(\Omega)}\leq C_1C_2^{-\gamma}$, then we have $w(x,0)\geq
u_0(x)$. Thus, we have shown that $C_1,C_2$ satisfy
\begin{equation}\label{C1C2}
\max\left\{\|u_0\|_{L^\infty(\Omega)}C_2^\gamma,(2\gamma)^\gamma\right\}\leq
(2\alpha)^\frac{1}{\beta-q+1}C_2^\gamma.
\end{equation}
In order (\ref{C1C2}) to be satisfied, we need
\begin{equation}\label{C1C21}
\|u_0\|_{L^\infty(\Omega)}\leq\epsilon:=(2\alpha)^\frac{1}{\beta-q+1}
\end{equation}
and
\begin{equation}\label{C1C22}
C_2\geq2^{\frac{1}{\gamma(q-\beta-1)}+1}\gamma.
\end{equation}

For $C_1,C_2$ satisfying (\ref{C1C2}) and (\ref{C1C22}), we know that $w$ is
a super-solution, which implies that
\begin{equation}
0\leq u\leq C_1(t+C_2)^{-\gamma},\gamma=\frac{1}{\beta-1},\mbox{for}\
1<\beta<q-1
\end{equation}
provided $u_0$ satisfies (\ref{C1C21}).

If $1<\beta=q-1$, assume additionally that $\alpha<1$, we can still obtain
the first estimate in (\ref{Decay}) for $C_1,C_2$ and $u_0$ satisfying
\begin{equation}
C_1\geq\max\left\{C_2^\gamma,\left(\frac{\gamma}{1-\alpha}\right)^\gamma\right\}.
\end{equation}

If $\beta=1,q\geq2$, let
\begin{equation}
w=C_1\me^{-C_2t}
\end{equation}
with
\begin{equation}
\left\{\begin{aligned}
&\alpha C_1^{q-2}+C_2\leq1,\|u_0\|_{L^\infty(\Omega)}\leq C_1,\mbox{for}\ q>2;\\
&0<C_2\leq1-\alpha,\|u_0\|_{L^\infty(\Omega)}\leq C_1,\mbox{for}\ q=2.
\end{aligned}\right.
\end{equation}
We can still verify that $w$ is a super-solution of \eqref{IBP}. Then we obtain the desired result by comparison principle. Thus, the proof is complete.
\end{proof}
\section{Finite time blowup}\setcounter{equation}{0}
In this section we will use two different methods to show that the solution of (\ref{IBP}) will blow up in finite time. We first introduce the following blowup result which is based on the construction of a self-similar sub-solution and the comparison principle.
\begin{thm}\label{LID}
Suppose that $q>\max\{p,2,\beta+1\}$. Then the solution of (\ref{IBP}) will blow up in finite time for some large $u_0(x)$ satisfying $u_0(x)>0$ in $\Omega'\subset\Omega$.
\end{thm}

\begin{proof} Without loss of generality, we assume that
$0\in\Omega$. Define $v(x,t)$ as:
\begin{equation}\label{DV}
v(x,t)=\frac{1}{(1-\varepsilon t)^k}V\left(\frac{|x|}{(1-\varepsilon
t)^m}\right),\ t_0\leq t<\frac{1}{\vep},
\end{equation}
where
\begin{equation}
V(y)=1+\frac{A}{\sigma}-\frac{y^\sigma}{\sigma A^{\sigma-1}}, y\geq0,
\end{equation}
and
\begin{equation}\label{cs}
\sigma=\frac{p}{p-1},k=\frac{1}{q-2},1<m<\frac{q-p}{p(q-2)},A>\frac{2k}{m},0<\vep<\frac{\alpha}{k\left(1+\frac{A}{\sigma}\right)}.
\end{equation}
Let
\begin{equation}
R=\left(A^{\sigma-1}(\sigma+A)\right)^\frac{1}{\sigma},\ D:=\left\{(x,t)\
\big|\ t_0\leq t<\frac{1}{\vep},|x|<R(1-\vep t)^m\right\},
\end{equation}
then $V(y)\geq0$ is smooth in $D$ and $v(y)<0$ if $y>R$. Moreover, $V(y)$
satisfies
\begin{equation}\label{PV}
\left\{\begin{array}{ll}
1\leq V(y)\leq 1+\D\frac{A}{\sigma},-1\leq V'(y)\leq 0,&\ \mbox{if}\ 0\leq
y\leq A;\\[3mm]
0\leq V(y)\leq 1,-\D\frac{R^{\sigma-1}}{A^{\sigma-1}}\leq V'(y)\leq-1,&\
\mbox{if}\ A\leq y\leq R;\\[3mm]
\left(|V'|^{p-2}V'\right)'+\D\frac{N-1}{y}|V'|^{p-2}V'=-\D\frac{N}{A}.&
\end{array}\right.
\end{equation}
Define
\begin{equation}
\mathcal{L}_pv=v_t-\Delta_pv-\alpha|v|^{q-2}v+|v|^{\beta-1}v,
\end{equation}
then
\begin{equation}
\mathcal{L}_pv=\frac{\vep(kV+myV')}{(1-\vep
t)^{k+1}}-\frac{\left(|V'|^{p-2}V'\right)'+\frac{N-1}{y}|V'|^{p-2}V'}{(1-\vep
t)^{(k+m)(p-1)+m}}-\frac{\alpha V^{q-1}}{(1-\vep
t)^{k(q-1)}}+\frac{V^\beta}{(1-\vep t)^{k\beta}}.
\end{equation}
By (\ref{cs}), we can easily see that
$k+1=k(q-1),k\beta<k+1,(k+m)(p-1)+m<k+1$. Then, for
$0\leq\frac{1}{\vep}-t_0\ll 1$ and $t_0\leq t<\frac{1}{\vep}$, if
$y\in[0,A]$,
\begin{equation}\label{SS1}
\begin{split}
\mathcal{L}_pv&=\frac{1}{(1-\vep
t)^{k+1}}\Bigg\{\vep(kV+myV')+\frac{N}{A}(1-\vep t)^{k+1-m-(k+m)(p-1)}-\alpha
V^{q-1}\\
&\quad\,+V^\beta(1-\vep t)^{k+1-k\beta}\Bigg\}\\
&\leq\frac{1}{(1-\vep t)^{k+1}}\Bigg\{\vep
k(1+\frac{A}{\sigma})+\frac{N}{A}(1-\vep t)^{k+1-m-(k+m)(p-1)}-\alpha\\
&\quad\,+V^\beta(1-\vep t)^{k+1-k\beta}\Bigg\}\\
&\leq 0,\ \mbox{for}\ \vep\ll\frac{\alpha}{k\left(1+\frac{A}{\sigma}\right)}.
\end{split}
\end{equation}
Similarly, if $y\in[A,R]$,
\begin{equation}\label{SS2}
\begin{split}
\mathcal{L}_pv&\leq\frac{1}{(1-\vep
t)^{k+1}}\left\{\vep(k-mA)+\frac{N}{A}(1-\vep t)^{k+1-m-(k+m)(p-1)}+(1-\vep
t)^{k+1-k\beta}\right\}\\
&\leq0.
\end{split}
\end{equation}
Thus, we have prove that $\mathcal{L}_pv\leq0$ in $D$. In order for $v(x,t)$ to be
a sub-solution, we also need to choose suitable initial data and
boundary value. Let $t_0$ be such that $u_0(x)>0$ in $B(0,R(1-\vep
t_0)^m)\subset\Omega$ and $u_0(x)\geq v(\cdot,t_0)$ in $B(0,R(1-\vep t_0)^m)$.
According to Theorem \ref{HOR} and the definition of $v,u(x,t)\geq0=v(x,t)$ in $\partial B(0,R(1-\vep t)^m)\times(t_0,\frac{1}{\vep})$. Thus, we have shown that $v(x,t+t_0)$ is a sub-solution for \eqref{IBP} in $D(t_0):=\{(x,t)|0\leq t\leq\frac{1}{\vep}-t_0,|x|<R(1-\vep(t+t_0))^m\}$. By Proposition \ref{CP},
\begin{equation}
u(x,t)\geq v(x,t+t_0),\ (x,t)\in D(t_0).
\end{equation}
Noticing that $\lim_{t\rightarrow1/\vep}v(0,t)\rightarrow+\infty$,
we have $u$ must blow up at a finite time
$T\leq\frac{1}{\vep}-t_0<\infty$.
\end{proof}

\begin{rem}
If $1<p<2$ we can also choose $m$ such that $0<m<\frac{2-p}{p(q-2)}$ in
(\ref{cs}).
\end{rem}

\begin{rem}
The method we used above is first introduced by Souplet and Weissler in
\cite{SW} for $p=2$. Li and Xie developed this method in \cite{LX} for
$p>2$. In our latest papers \cite{ZL,ZL1}, we used this method to study the blowup
results of the initial boundary problem for a p-Laplacian  parabolic equation
with a nonlinear gradient term.
\end{rem}
Next, we will introduce some blowup results whose proofs are based on the
energy method and concavity method which were also used in \cite{AD,LX,YJ,JNZ}
and the references therein. In the proof of our desired results, the
following lemma concerning the so-called ``energy'' is useful.

\begin{lem}\label{NE}
Let
\begin{equation}
E(t)=\int_\Omega\left(\frac{1}{p}|\nabla
u|^p+\frac{1}{\beta+1}u^{\beta+1}-\frac{\alpha}{q}u^q\right)\ \dif x.
\end{equation}
If $E(0)<0$, then $E(t)<0$ for all $t>0$.
\end{lem}

\begin{proof} By a direct computation, we can see that
\begin{equation}
\begin{split}
E'(t)&=\int_\Omega\left(|\nabla u|^{p-2}\nabla u\cdot\nabla u_t+u^\beta
u_t-\alpha u^{q-1}u_t\right)\ \dif x\\
&=\int_\Omega\left(-\Delta_pu+u^\beta-\alpha u^{q-1}\right)u_t\ \dif x
=-\int_\Omega u_t^2\ \dif x\leq 0.
\end{split}
\end{equation}
Hence, $E(t)\leq E(0)<0$ for all $t>0$.
\end{proof}
The following theorem is the main result of this section.
\begin{thm}\label{BU}
Suppose $u_0(x)$ satisfies
\begin{equation}\label{E0}
\int_\Omega\left(\frac{1}{p}|\nabla
u_0|^p+\frac{1}{\beta+1}u_0^{\beta+1}-\frac{\alpha}{q}u_0^q\right)\ \dif x<0.
\end{equation}
then the solution of (\ref{IBP}) will blow up in finite time provided that
one of the following cases occurs:

$\mathrm{(a)}$ $0<\beta<\min\{1,p-1\},q>\max\{p,2\}$;

$\mathrm{(b)}$ $q=p,1<\beta<p-1$;

$\mathrm{(c)}$ $\beta=p-1,q>\max\{p,2\}$;

$\mathrm{(d)}$ $1<\beta<p-1,q>p>2$;

$\mathrm{(e)}$ $\beta+1=q=p>2$.

$\mathrm{(f)}$ $q>\beta+1>p>2$, and $\|u_0\|_{L^2(\Omega)}^2$ is large
enough.
\end{thm}
\begin{proof}Let $y(t)=\|u\|_{L^2(\Omega)}^2$, then it
satisfies
\begin{equation}
\frac{1}{2}y'(t)=\int_\Omega uu_t\ \dif
x=\int_\Omega\left(u\Delta_pu-u^{\beta+1}+\alpha u^q\right)\ \dif x
=\int_\Omega\left(-|\nabla u|^p-u^{\beta+1}+\alpha u^q\right)\ \dif x.
\end{equation}
By Lemma \ref{NE}, we can get
\begin{equation}\label{EODI}
\begin{split}
\frac{1}{2p}y'(t)&=\int_\Omega\left(-\frac{1}{p}|\nabla
u|^p-\frac{1}{p}u^{\beta+1}-\frac{\alpha}{p}u^q\right)\ \dif x\\
&=-E(t)+\left(\frac{1}{\beta+1}-\frac{1}{p}\right)\int_\Omega u^{\beta+1}\
\dif x+\alpha\left(\frac{1}{p}-\frac{1}{q}\right)\int_\Omega u^q\ \dif x\\
&>\left(\frac{1}{\beta+1}-\frac{1}{p}\right)\int_\Omega u^{\beta+1}\ \dif
x+\alpha\left(\frac{1}{p}-\frac{1}{q}\right)\int_\Omega u^q\ \dif x.
\end{split}
\end{equation}
Let us now estimate (\ref{EODI}) furthermore in different cases.

$\mathrm{(a)}$. $0<\beta<\min\{1,p-1\},q>\max\{p,2\}$. In this case, by
H\"{o}lder's inequality, (\ref{EODI}) can be rewritten as
\begin{equation}
\frac{1}{2p}y'(t)\geq\alpha\left(\frac{1}{p}-\frac{1}{q}\right)\int_\Omega
u^q\ \dif x
\geq\alpha\left(\frac{1}{p}-\frac{1}{q}\right)|\Omega|^\frac{2-q}{2}y^\frac{q}{2},
\end{equation}
i.e.
\begin{equation}\label{OI}
y'(t)\geq
2p\alpha\left(\frac{1}{p}-\frac{1}{q}\right)|\Omega|^\frac{2-q}{2}y^\frac{q}{2}(t).
\end{equation}
Integrating (\ref{OI}) in $t$, we have
\begin{equation}
y(t)\geq\left(y^\frac{2-q}{2}(0)-p\alpha(q-2)\left(\frac{1}{p}-\frac{1}{p}\right)|\Omega|^\frac{2-q}{2}t\right)^\frac{2}{2-q}
\end{equation}
which implies that
\begin{equation}
y(t)\rightarrow+\infty,\ \mbox{as}\ t\rightarrow
T_1^*:=\frac{y^\frac{2-q}{2}(0)|\Omega|^\frac{2}{2-q}}{p\alpha(q-2)\left(\frac{1}{p}-\frac{1}{q}\right)}.
\end{equation}

$\mathrm{(b)}$. $q=p,1<\beta<p-1$. In this case, there holds
\begin{equation}
y'(t)>2p\left(\frac{1}{\beta+1}-\frac{1}{p}\right)\int_\Omega u^{\beta+1}\
\dif
x\geq2p\left(\frac{1}{\beta+1}-\frac{1}{p}\right)|\Omega|^\frac{1-\beta}{2}y^\frac{\beta+1}{2}(t).
\end{equation}
Then
\begin{equation}
y(t)\geq\left(y^\frac{1-\beta}{2}(0)-p(\beta-1)\left(\frac{1}{\beta+1}-\frac{1}{p}\right)|\Omega|^\frac{1-\beta}{2}t\right)^\frac{2}{1-\beta}.
\end{equation}
Thus
\begin{equation}
y(t)\rightarrow+\infty,\ \mbox{as}\ t\rightarrow
T_2^*:=\frac{y^\frac{1-\beta}{2}(0)|\Omega|^\frac{2}{1-\beta}}{p(\beta-1)\left(\frac{1}{\beta+1}-\frac{1}{p}\right)}.
\end{equation}

$\mathrm{(c)}$. $\beta=p-1,q>\max\{p,2\}$. Similarly as (1), we can derive
that $y(t)\rightarrow+\infty$, as $t\rightarrow T_3^*=T_1^*$.

$\mathrm{(d)}$. $1<\beta<p-1,q>p>2$. We can rewrite (\ref{EODI}) as
\begin{equation}
\begin{split}
y'(t)&\geq2p\left(\frac{1}{\beta+1}-\frac{1}{p}\right)|\Omega|^\frac{1-\beta}{2}y^\frac{\beta+1}{2}(t)+2p\alpha\left(\frac{1}{p}-\frac{1}{q}\right)
|\Omega|^\frac{2-q}{2}y^\frac{q}{2}(t)\\
&\geq4p\sqrt{\alpha\left(\frac{1}{\beta+1}-\frac{1}{p}\right)\left(\frac{1}{p}-\frac{1}{q}\right)}|\Omega|^\frac{3-q-\beta}{4}y^\frac{q+\beta+1}{4}(t).
\end{split}
\end{equation}
Then $y(t)\rightarrow+\infty$, as $t\rightarrow
T_4^*\leq\min\{T_1^*,T_2^*,T'\}$ with
\begin{equation}
T':=\frac{y^\frac{3-q-\beta}{4}(0)|\Omega|^\frac{4}{3-q-\beta}}{p(q+\beta-3)\sqrt{\alpha\left(\frac{1}{\beta+1}-\frac{1}{p}\right)\left(\frac{1}{p}
-\frac{1}{q}\right)}}.
\end{equation}

$\mathrm{(e)}$. $\beta+1=q=p$. If this happens, then we can only derive from
(\ref{EODI}) that $y'(t)>0$ which can not be used to show that
$y(t)\rightarrow+\infty$ as $t\rightarrow\widetilde{T}<\infty$. However, if
$p>2$, we can still obtain desired result by the concavity method. The proof
here is same as the one of \cite[Lemma 3.4]{LX}, here we just provided the
final ordinary inequality below:
\begin{equation}
y'(t)\geq\frac{y'(0)}{y^\frac{p}{2}(0)}y^\frac{p}{2}(t).
\end{equation}

$\mathrm{(f)}$. $q>\beta+1>p>2$. As $\beta+1>p$, the first term of the right
side hand in (\ref{EODI}) is negative, we cannot use the procedure above
directly. However, by the fact that $q>\beta+1$, we can still obtain the
desired result. Indeed, by Young's inequality, we have for small $\epsilon>0$
\begin{equation}
\int_\Omega u^{\beta+1}\ \dif x\leq\frac{\epsilon(\beta+1)}{q}\int_\Omega
u^q\ \dif x+C(\epsilon)\frac{q-\beta-1}{q}|\Omega|.
\end{equation}
Choose a suitable $\epsilon$ such that
\begin{equation}
\left(\frac{1}{\beta+1}-\frac{1}{p}\right)\frac{\epsilon(\beta+1)}{q}\geq-\frac{\alpha}{2}\left(\frac{1}{p}-\frac{1}{q}\right).
\end{equation}
Then we have
\begin{equation}
\frac{1}{2p}y'(t)\geq\frac{\alpha}{2}\left(\frac{1}{p}-\frac{1}{q}\right)\int_\Omega
u^q\ \dif
x+C(\epsilon)\left(\frac{1}{\beta+1}-\frac{1}{p}\right)\frac{q-\beta-1}{q}|\Omega|.
\end{equation}
If we assume additionally that $\|u_0\|_{L^2(\Omega)}^2$ is large enough,
then we can derive
\begin{equation}
y'(t)\geq\frac{p\alpha}{2}\left(\frac{1}{p}-\frac{1}{q}\right)|\Omega|^\frac{2-q}{2}y^\frac{q}{2}(t)
\end{equation}
which implies that
\begin{equation}
y(t)\rightarrow+\infty,\ \mbox{as}\ t\rightarrow T_5^*:=4T_1^*.
\end{equation}

The proof of Theorem \ref{BU} is now complete.
\end{proof}

\begin{rem}
Following the same manner as in \cite[Theorem 3.5]{LX}, we can still obtain
the desired blowup results in case $\mathrm{(e)}$ of Theorem \ref{BU} if we
assume that $\alpha>\Lambda_1+1$ instead of (\ref{E0}).
\end{rem}
\begin{rem}
During the proof of Theorem \ref{BU}, we also obtain an upper bound of the blowup time in each case.
\end{rem}

\section{Discussions}\setcounter{equation}{0}
As was shown in the previous sections, the relation of $p,q,\beta$ plays an
important role in determining the properties of the weak solution of
(\ref{IBP}). To be specific, we will state it for $1<p<2$ and $p>2$
respectively. Moreover, we will use two figures to state the results of
blowup, extinction and global existence intuitionally. For simplicity, we
will not point out which domain the boundary lines and the coordinate axis
belong to.

\begin{figure}[!ht]
\begin{center}
\setlength{\unitlength}{1mm}
\begin{picture}(150,75)
\put(8,2){0}
\put(10,5){\vector(1,0){90}} \put(98,7){$ q $}             
\put(10,5){\vector(0,1){60}}  \put(11,65){$ \beta $}      
\multiput(20,5)(10,0){4}{\circle*{1}}
\multiput(10,15)(0,10){4}{\circle*{1}}
\put(30,5){\line(1,1){10}} \put(82,62){$ q=\beta+1 $}    
\put(16,1){\small $ p-1 $}  \put(29,1){1} \put(49,1){2} \put(39,1){$ p $}
\put(0,14){\small $ p-1 $}  \put(7,24){1} \put(7,44){2} \put(7,34){$ p $}
\put(40,5){\line(0,1){60}}
\put(15,25){\shortstack{\large Global\\[5mm] \large Existence}}
\put(60,23){\shortstack{\large Extinction(E2)\\[5mm] \large or\\[5mm] \large
Blowup}}
\linethickness{1.5pt}
\put(50,5){\line(0,1){60}}
\qbezier[10](40,15)(62.5,42.5)(85,55)
\put(35,7){E1}  \put(43,33){E2}
\put(105,25){\shortstack{E1: Extinction for\\[3mm]any initial data\\[5mm] E2:
Extinction for\\[3mm]small initial data}}
\end{picture}
\caption{$1<p<2$}\label{fig1}
\end{center}
\end{figure}
We first discuss the case $1<p<2$ ({\bf Figure \ref{fig1}}). In this case, if
$q>\max\{2,\beta+1\}$ or
$0<\beta\leq p-1,\ q>2$, then finite time blowup will occur for some suitably
large initial data, {see Theorem \ref{LID} and \ref{BU}((a),(c)).} If $q\in(\beta+1,p)$, or $q=\beta+1$, or $q=p$, then
finite time extinction will happen with suitable $\alpha$ and any nontrivial
initial data, see {Theorem \ref{FTE}.} While if $q>p,\beta>0$, then small initial data can lead to
finite time extinction, {see Theorem \ref{FTE2}.} Noticing that if $q>2$, then large initial data can
lead to finite time blowup while small initial data implies finite time
extinction which is interesting.

\begin{figure}[!ht]
\begin{center}
\setlength{\unitlength}{1mm}
\begin{picture}(110,70)
\put(8,2){0}
\put(10,5){\vector(1,0){90}} \put(98,7){$ q $}             
\put(10,5){\vector(0,1){60}}  \put(12,63){$ \beta $}        
\put(40,5){\circle*{1}}  \put(36,1){\small $ p-1 $}       
\put(30,5){\circle*{1}}  \put(29,1){1}                      
\put(50,5){\circle*{1}} \put(49,1){2}                       
\put(60,5){\circle*{1}} \put(59,1){$ p $}                 
\put(10,35){\circle*{1}} \put(0,34){\small $ p-1 $}        
\put(10,25){\circle*{1}} \put(7,24){1}                       
\put(10,45){\circle*{1}} \put(7,44){2}                       
\put(10,55){\circle*{1}} \put(7,54){$ p $}                 
\put(30,5){\line(1,1){55}} \put(82,62){$ q=\beta+1 $}    
\put(60,5){\line(0,1){30}}                                    
\put(50,25){\line(1,0){50}}
\qbezier[50](10,25)(30,25)(50,25)
\put(20,35){\Large Global Existence}         
\put(75,13){\Large Blowup}
\put(40,10){Extinction}
\put(75,30){\shortstack{\Large Decay\\or\\Blowup}}
\put(53,26){\shortstack{\tiny Decay}}
\end{picture}
\caption{$p>2$}\label{fig2}
\end{center}
\end{figure}
Next, let us consider the case $p>2$ ({\bf Figure \ref{fig2}}). In this
case, if $q>\max\{p,\beta+1\}$,
or $q>p,\ 0<\beta\leq q-1$, or $q=p\geq\beta+1>2$, then for some suitably
large initial data, the solution of (\ref{IBP}) will blow up in finite time, {see Theorem \ref{LID} and Theorem \ref{BU}((b),(d),(e),(f)).}
If $q\in(\beta+1,p),\ \beta<1$, or $q=\beta+1<2$, or $q=p$, then finite time
extinction will happen with suitable $\alpha$ and any nontrivial initial
data, see Theorem \ref{FTE}. Besides, if $1\leq\beta\leq q-1$, then as was shown in Theorem
\ref{decay}, the solution of (\ref{IBP}) cannot extinction in finite time,
while it will decay to zero as $t\rightarrow+\infty$ for some suitably small
$u_0$.

We also need to point out that finite time extinction is not a singularity
property for solution of (\ref{IBP}) as $\beta$ and $q-1$ are positive. If
finite time extinction happens, we have in fact shown that the solution of
(\ref{IBP}) is global in time bounded which is also an important property of
the solution of (\ref{IBP}). For the global existence of the weak solution,
we can see from Theorem \ref{GE} that the critical value for $q$ is $p$ if
$1<p<2$. While in the degenerate case, the critical value is $p$ and
$\beta+1$. Moreover, if $q\leq p$ or $q<\beta+1$, then we can obtain the
global existence.
\section*{Acknowledgement}
This work was supported in part
by the National Natural Science Foundation of China (No.
11371286, 11401458), the Special Fund of
Education Department (No. 2013JK0586) and the Youth Natural Science Grant
(No. 2013JQ1015)
of Shaanxi Province of China.

%

\end{document}